\documentclass[reqno, 12pt]{amsart}
\usepackage{amsmath,amsfonts,amsthm,amscd,upref,amstext}
\usepackage{amssymb}
\usepackage{subfigure}
\usepackage[all]{xy}
\usepackage{comment}
\usepackage{xcolor}
\usepackage{fullpage}

\newtheorem{prop}{Proposition}[section]
\newtheorem{thm}[prop]{Theorem}
\newtheorem{cor}[prop]{Corollary}

\newtheorem{defn}[prop]{Definition}
\newtheorem{rem}[prop]{Remark}
\newtheorem{lem}[prop]{Lemma}

\newcommand{\N}{\mathbb{N}}

\newcommand{\m}{\mathfrak{m}}
\newcommand{\ia}{\mathfrak{a}}
\newcommand{\ib}{\mathfrak{b}}

\newcommand{\p}{\mathfrak{p}}
\newcommand{\q}{\mathfrak{q}}
\newcommand{\s}{\mathcal{S}}

\newcommand{\h}{\mathrm{H}}
\newcommand{\Ass}{\mathrm{Ass}}

\DeclareMathOperator{\Hom}{Hom}
\DeclareMathOperator{\Ext}{Ext}
\DeclareMathOperator{\Tor}{Tor}
\DeclareMathOperator\Att{Att}

\DeclareMathOperator\ass{Ass}
\DeclareMathOperator{\pdim}{pd}

\DeclareMathOperator{\Supp}{Supp}

\DeclareMathOperator{\cdim}{cd}

\numberwithin{equation}{section}

\begin{document}

\title{On generalized Hartshorne's conjecture and local cohomology modules}


\date{}

\author{T. H. Freitas}
\address{Universidade Tecnol\'ogica Federal do Paran\'a, 85053-525, Guarapuava-PR, Brazil}
\email{freitas.thf@gmail.com}

\author{V. H. Jorge P\'erez}
\address{Universidade de S{\~a}o Paulo -
ICMC, Caixa Postal 668, 13560-970, S{\~a}o Carlos-SP, Brazil}
\email{vhjperez@icmc.usp.br}

\author{L. C. Merighe*}
\address{Universidade de S{\~a}o Paulo -
ICMC, Caixa Postal 668, 13560-970, S{\~a}o Carlos-SP, Brazil}
\email{liliam.merighe@gmail.com}

\date{\today}
\thanks{The first and second author were  supported  CNPq-Brazil 421440/2016-3.}
\thanks{The third author was  supported by CNPq-Brazil - Grants 142376/2016-7.}
\thanks{* Corresponding author.}

\begin{abstract} 
Let $\mathfrak{a}$ denote an ideal of a commutative Noetherian ring $R$. Let $M$ and $N$ be  two $R$-modules. In this paper, we give some answers on the extension of Hartshorne's conjecture about the cofiniteness of torsion and extension functors. For this purpose, we study the cofiniteness  of the generalized local cohomology module $\h^{i}_{\mathfrak{a}}(M,N)$ for a new  class of modules, called $\mathfrak{a}$-weakly finite modules, in the local and non-local case. Furthermore, we derive some results on attached primes of  top generalized local cohomology modules. 
\end{abstract}

\maketitle

2010 Matahematics Subject Classificafiton: 13D45, 13H10, 13H15, 13E10

\newpage

\section{Introduction}

\hspace{0.5cm}

Throughout  this paper,  let $\ia$ and $\ib$ be ideals of a  commutative Noetherian ring with nonzero identity $R$. If $M$ and $N$ are two $R$-modules and $i\in \mathbb{Z}$, consider
$$\h^{i}_{\mathfrak{a}}(M,N):=\displaystyle \varinjlim_n{\rm Ext}_R^{i}(M/\mathfrak{a}^nM,N),$$
the $i$th generalized local cohomology module with respect to $\mathfrak{a}$ and $R$-modules $M$ and $N$, introduced by Herzog  \cite{herzog}. 
If $M=R$ the definition of generalized local cohomology modules  reduces to the notion introduced by Grothendieck of local cohomology modules $\h^{i}_{\mathfrak{a}}(N)$
\cite{grot}.

One of the interesting open questions in local cohomology theory is the following:

\

{\bf Question 1:} When is ${\rm Ext}^i_R(R/\mathfrak{a}, \h^j_{\mathfrak{a}}
(N))$ finitely generated for all integer $i$ and $j$?

\

This question was proposed by Hartshorne \cite{HA} who, in turn, was motivated by the conjecture made by Grothendieck \cite{grotpaper} about the finiteness of ${\rm Hom}_R(R/\mathfrak{a}, \h^j_{\mathfrak{a}}
(N))$. Also, Hartshorne introduced  an interesting class of modules called $\mathfrak{a}$-cofinite modules. Recall that an $R$-module $M$ is said to be $\mathfrak{a}$-cofinite if ${\rm Supp}(M)\subseteq V(\mathfrak{a})$ and ${\rm Ext}^i_R(R/\mathfrak{a}, M)$ is finitely generated for all $i$.

In this sense, as a generalization of Hartshorne's 
conjecture, we have a natural question.

\

{\bf Question 2:} When is 
 $ \h^i_{\mathfrak{a}}
(M,N)$ $\mathfrak{a}$-cofinite for all $i$?  

\

Concerning this question, several results were obtained, and in most of them the finiteness assumption on the modules is crucial in the proof (for example \cite{goto}, \cite{divsa} and \cite{tornovo}).

In this paper, we will introduce a new class of modules, called $\ia$-weakly finite modules over $M$ (see Defininiton \ref{DefWF}). This class is an extension of the class of weakly finite modules introduced in \cite{B} (which contains the class of finitely generated modules, big Cohen-Macaulay modules and $\mathfrak{a}$-cofinite modules). In the local and non-local case we will show the following two results, that improves \cite[Theorem 2.2]{fdehg} and \cite[Theorem 2.2]{DST} respectively.

\begin{thm}\label{1} 
	Suppose that $\ia \subseteq \ib$ and $\dim_R R/\ib = 0$. Let $M$ be a finitely generated $R$-module and $N$ be an $\ia$-weakly finite $R$-module over $M$. Then $\h_{\ib}^i (M, N)$ is an Artinian $R$-module and $\ia$- and $\ib$-cofinite, for all $i \in \N$.
\end{thm}

\begin{thm}\label{2}
Let $(R, \m)$ be a local ring. Consider $M$, $N$ two $R$-modules such that $M$ is finitely generated  of finite projective dimension $(\pdim M = d < \infty)$, $N$ is $\ib$-weakly finite over $M$ and $\dim_R N = n < \infty$. Suppose that $\ia \subseteq \ib$. Then $\h_{\ia}^{d+n}(M, N)$ is Artinian and $\ib$-cofinite.

\noindent In particular, if $\ia = \ib$, then $\h_{\ia}^{d+n}(M, N)$ is Artinian and $\ia$-cofinite.
\end{thm}

\

A non-zero $R$-module $L$ is called \textit{secondary} if its multiplication map by any element $x$ of $R$ is either surjective or nilpotent. A \textit{secondary representation} for a $R$-module $L$ is an finite expression $L= L_1+L_2+ \ldots+L_s$,  where  $L_i$ is secondary for $1\leq i\leq s$. We will say that $L$ is \textit{representable} if there exist such an expression. 

Furthermore, a prime ideal $\mathfrak{p}$ of $R$ is said to be an {\it attached prime} of $L$, if $\mathfrak{p}=(K:_RL)$ for some submodule $K$ of $L$. Denote  by ${\rm Att}_R(L)$ the set of attached prime ideals of the $R$-module $L$. If $L$ admits a secondary representation $L= L_1+L_2+ \ldots+L_s$, then ${\rm Att}_R(L)$ is exactly the set $\{ \sqrt{(0:_RL_i)} \mid  1\leq i\leq s \}$. We can say that such representation is reduced if $\sqrt{(0:_RL_i)}$, $i=1, \ldots, s$, are all distinct.  For more details, the reader can see  \cite{macd}.

Another aim of the present paper is to show some results concerning attached primes of generalized local cohomology modules $\h^i_{\mathfrak{a}}(M,N)$. More precisely,
we will show an extension of the main result in \cite{DiYa}.

\begin{thm}\label{3}
Let $M$ be a finitely generated $R$-module such that $\pdim M = d < \infty$. Let $N$ be an $R$-module such that $\dim_R N = \dim_R R = n$. Then $\h_{\ia}^{d+n}(M,N)$ has secondary representation and
$$\Att_R (\h_{\ia}^{d+n}(M,N)) \subseteq \{ \p \in \ass_R (N) \mid \cdim (\ia, M, R/\p) = d+n \}.$$
\end{thm}
\noindent Remerber that $\cdim (\ia, M, N) = \sup \{ i \mid \h^i_{\ia}(M,N) \neq 0\}$ for any $M$ and $N$ $R$-modules.

Returning to Question 1, Melkersson \cite[Theorem 2.1]{M1} has shown that ${\rm Ext}^j_R(R/\mathfrak{a}, M)$ is finitely generated for all $j \in \N$ if and only if ${\rm Tor}_j^R(R/\mathfrak{a}, M)$ is finitely generated for all $j \in \N$. We generalize this result on Theorem \ref{teo1} for a Serre subcategory of the category of $R$-modules.

We can analyze Question 1 from a more general point of view.

\

\noindent \textbf{Question 3:} When are ${\rm Ext}_R^i(M,N)$ and ${\rm Tor}^R_i(M,N)$  $\mathfrak{a}$-cofinite (or finite length, or $\mathfrak{a}$-cominimax, or $\mathfrak{a}$-weakly cofinite)   for all (or  for some) integer $i$? 

\

Note that $\mathfrak{a}$-cominimax \cite{azami} and $\mathfrak{a}$-weakly cofinite modules \cite{divmafi2} 
are generalizations of $\mathfrak{a}$-cofinite modules (see definitions on Section 4). 
We may also ask, as a particular case of Question 3, when $ \Ext^j_R\left(L,\h^i_{\mathfrak{a}}
(M,N)\right)$ and  
 $ \Tor_j^R\left(L,\h^i_{\mathfrak{a}}
(M,N)\right)$ are $\mathfrak{a}$-cofinite (or finite length, or $\mathfrak{a}$-cominimax, or $\mathfrak{a}$-weakly cofinite)  for all (or for some) integers $i$ and $j$, and some $R$-module $L$.  In this sense, we show the following result.

\begin{thm}
Let $(R, \m)$ be a commutative Noetherian local ring. Let $H$ and $M$ be $R$-modules such that $H$ is Artinian and $\mathfrak{a}$-cofinite and $M$ is minimax. Then, for each $i \geq 0$, the module $\Ext_R^i(M,H)$ is minimax and $\widehat{\mathfrak{a}}$-cofinite over $\widehat{R}$. 
\end{thm}

One of the main goals of this paper is to generalize the results showed by  Naghipour, Bahmanpour and Khalili  \cite{arttor}. In this direction we show the following theorem:

\begin{thm} Let $N$ be a nonzero $\mathfrak{a}$-cominimax $R$-module and $M$ be a finitely generated $R$-module.
\begin{itemize}
\item[(i)] If $\dim_R M=1$, then the $R$-module ${\rm Tor}^R_i(M,N)$ is Artinian and $\mathfrak{a}$-cofinite for all $i\geq 0$. 

\item[(ii)] If $\dim_R M=2$, then the $R$-module ${\rm Tor}^R_i(M,N)$ is $\mathfrak{a}$-cofinite for all $i\geq 0$. 
\end{itemize}
\end{thm}

For larger dimensions we obtain the following result.

\begin{thm} Let $(R,\mathfrak{m})$ be a local ring and $\mathfrak{a}$ an $R$-ideal. Let $N$ be a nonzero $\mathfrak{a}$-cominimax $R$-module and $M$ be a finitely generated $R$-module. 
 Then the $R$-module ${\rm Tor}^R_i(M,N)$ is $\mathfrak{a}$-weakly cofinite for all $i\geq 0$ when one of the following cases holds:
 \begin{itemize}
 \item[(i)] $\dim_R N\leq 2$.
 
 \item[(ii)] $\dim_R M=3$.
 \end{itemize} 
\end{thm}

This paper is organized as follows. In Section 2, we discuss Question 2, define the class of $\ia$-weakly finite modules over $M$ and prove some results with the purpose to show Theorem \ref{1} and Theorem \ref{2}. 

In Section 3, we obtain some results on attached primes of generalized local cohomology in order to  prove our second main result (Theorem \ref{3}). 

In order to understand the cofiniteness behavior of extension and torsion functors,  in Section 4 we show some preliminary results that will be useful for this purpose (being the Theorem \ref{teo1} the most important of them).  Also, we recall the definitions of minimax, weakly Laskerian modules, $\ia$-cominimax and $\ia$-weakly cofinite modules. These notions improve the definitions of Noetherian (and Artinian) modules and $\ia$-cofinite modules. 

In last section, we show some cases where ${\rm Tor}^R_i(M,N)$ is $\ia$-cofinite or $\ia$-weakly cofinite module for all $i$. As applications of the results shown in the Section 2, we study the behavior of $ \Ext^j_R\left(L,\h^i_{\mathfrak{a}}
(M,N)\right)$ and $ \Tor_j^R\left(L,\h^i_{\mathfrak{a}}
(M,N)\right)$, for some $R$-module $L$.

For conventions of notation,  basic results, and terminology  not given in this paper, the reader should consult the books Brodmann-Sharp \cite{grot} and Matsumura  \cite{Mats}.

\section{Cofiniteness and Artinianness of Generalized Local Cohomology Modules}

In this section, we fix our notation and list several results for the convenience of the reader.
Let $R$ be a commutative Noetherian ring. Throughout this paper, unless otherwise noted, the $R$-module  $N$ is not necessarily finitely generated and we denote the $R$-dimension of $N$ by $\dim_R N := \sup \{ \dim_R R/\p \mid \p \in \Supp_R (N) \}$.

\begin{defn} \label{DefWF}
Let $M$ be an $R$-module, and let $\ia$ be an ideal of $R$. Let $\mathcal{W}$ be the largest class of $R$-modules, i.e., the union of all such classes, satisfying the following four properties:
\begin{itemize}
\item[1.] If $N \in \mathcal{W}$, then $\Hom_R (R/\ia, N)$ and $\Hom_R (M/\ia M, N)$ are finitely generated.

\item[2.] If $N$ is a non-zero element of $\mathcal{W}$ and $x$ is a regular element of $R$, then $N \neq xN$, $N/xN \in \mathcal{W}$ and $\dim_R N/xN = \dim_R N - 1$.

\item[3.] If $N \in \mathcal{W}$, then $|\Ass_R (N)|<\infty$.

\item[4.] If $N \in \mathcal{W}$, then $N/\Gamma_{\ia}(N) \in \mathcal{W}$.
\end{itemize}
	We say an $R$-module $N$ is $\ia$-weakly finite over $M$, if it belongs to $\mathcal{W}$.

If $M=R$, we say that the $R$-module is $\ia$-weakly finite. If $M=R$, $(R,\m)$ is a local ring and $\ia = \m$ we just say that the $R$-module is weakly finite (see \cite[Definition 2.1]{B}).
\end{defn}

Note that if $M$ is a finitely generated $R$-module, then $M$ is a weakly finite $R$-module. Another important class of modules is the class of $\ia$-cofinite modules. Recall that  a module $N$ is said to be $\mathfrak{a}$-cofinite \cite{HA} if ${\rm Supp}_R(N) \subseteq  V(\mathfrak{a})$ and $\Ext^i_R(R/\mathfrak{a},N)$ is a finitely generated module for all $i \in \N_0$. With this notion, Bagheri (\cite[Lemma 2.1]{B}) has shown, in the local case, that the class of $\ia$-cofinite modules is contained in the class of weakly finite modules. Our next proposition shows a similar behaviour for $\mathfrak{a}$-weakly finite modules, and the proof is analogous to what was done in \cite[Lemma 2.1]{B}.



\begin{prop} 
Let $(R, \mathfrak{m})$ be a local ring and let $M$ be a non-zero $R$-module of dimension $n > 0$. If  $M$ is $\ia$-cofinite, then $M$ is $\ia$-weakly finite.
\end{prop}







Also, Bagheri \cite[Theorem 2.1]{B} has shown in the local case  that the local cohomology module $\h_{\m}^{i}(N)$ is Artinian when $N$ is a weakly finite $R$-module.  Our main goal in this section is to generalize this result in the non-local case (Theorem \ref{Teo}).
To do this, we will need some previous results.

We call a Serre subcategory of the category of $R$-modules a class of $R$-modules closed under taking submodules, quotients and extensions (see more details and examples of this definition in Section 4).

\begin{lem} [{\cite[Corollary 4.4]{M1}}]\label{OBS}
The class of Artinian $\ia$-cofinite is a Serre subcategory of the category of $R$-modules.
\end{lem}


\begin{prop}\label{lema1} \label{Lema3N}
	Let $\mathcal{S}$ denote a Serre subcategory of the category of $R$-module. Let $M$ be a finitely generated $R$-module and $N\in \mathcal{S}$ any $R$-module. Then for all $i\geq 0$, ${\rm Ext}_R^i(M,N)$ and ${\rm Tor}_i^R(M,N)$ are elements of $\mathcal{S}$.
\end{prop}
\begin{proof}
	The proof follows from the definition of extension and torsion functors.
\end{proof}

\begin{lem} [{\cite[Lemma 2.1(i)]{DST}}] \label{Rem2.6N}
Let $M$ and $N$ $R$-modules such that $M$ is a finitely generated, then
	$$\h_{\ia}^i (M, \Gamma_{\ia}(N)) \cong \Ext_R^i (M, \Gamma_{\ia}(N)), \quad \mbox{ for all } i \in \N_0.$$
\end{lem}


\begin{lem} [{\cite[Lemma 2.1.1]{grot}}]\label{Obs2N}
	If $T$ is an $R$-module such that $|\Ass_R(T)|<\infty$, then $T$ is $\ia$-torsion free if and only if there is an element $r \in \ia$ such that $r$ is $T$-regular.
\end{lem}


\begin{prop} [{\cite[Proposition 4.1]{M1}}] \label{Obs3N}
	Let $M$ be an $R$-module with support in $V(\ia)$. Then $M$ is Artinian and $\ia$-cofinite if and only if $(0:_{M} \ia)$ has finite length. If there is an element $x \in \ia$ such that $(0:_{M} x)$ is Artinian and $\ia$-cofinite, then $M$ is Artinian and $\ia$-cofinite. 
\end{prop}

\begin{lem} \label{LEMMA}
	Suppose that $\ia \subseteq \ib$ and $\dim_R R/\ib = 0$. Let $N$ be an $\ia$-weakly finite $R$-module. Then $\Gamma_{\ib}(N)$ is Artinian and $\ia$- and $\ib$-cofinite.
\end{lem}
\begin{proof}
Note that $(0 :_N \ia) = \Hom_R(R/\ia, N)$ is finitely generated, since $N$ is $\ia$-weakly finite. Since $\ia \subseteq \ib$, $\Hom_R(R/\ib, N) = (0 :_N \ib) \subseteq (0 :_N \ia)$ and, since $R$ is a Noetherian ring, $\Hom_R(R/\ib, N)$ is finitely generated. Moreover, applying functors $\Hom_R(R/\ia, -)$ and $\Hom_R(R/\ib, -)$ to the exact sequence 
$$0 \rightarrow \Gamma_{\ib}(N) \rightarrow N$$
and using the same argument, we conclude $\Hom_R(R/\ia, \Gamma_{\ib}(N))$ and $\Hom_R(R/\ib, \Gamma_{\ib}(N))$ are finitely generated $R$-modules.

Since $\dim_R R/\ib = 0$, it follows that $\Hom_R(R/\ib, \Gamma_{\ib}(N)) = (0 :_{\Gamma_{\ib}(N)} \ib)$ has finite length. Thus, by Proposition \ref{Obs3N}, $\Gamma_{\ib}(N)$ is Artinian and $\ib$-cofinite. Hence, $\Hom_R(R/\ia, \Gamma_{\ib}(N)) = (0 :_{\Gamma_{\ib}(N)} \ia)$ is Artinian and therefore, has finite length. Then, again by Proposition \ref{Obs3N}, $\Gamma_{\ib}(N)$ is also $\ia$-cofinite.  
\end{proof}

Now we are able to show the first main result of this section.

\begin{thm} \label{Teo}
Suppose that $\ia \subseteq \ib$ and $\dim_R R/\ib = 0$. Let $M$ be a finitely generated $R$-module and let $N$ be a $\ia$-weakly finite $R$-module over $M$. Then $\h_{\ib}^i (M, N)$ is an Artinian $R$-module and is $\ia$- and $\ib$-cofinite, for all $i \in \N_0$.
\end{thm}
\begin{proof}
We use induction on $i$. Assume that $i=0$. We claim $\h_{\ib}^0(M,N)$ is Artinian and $\ia$- and $\ib$-cofinite.
Since $\h_{\ib}^0(M,N) \cong \Hom_R(M, \Gamma_{\ib}(N))$ the assertion follows from Lemma \ref{LEMMA} and Proposition \ref{Lema3N}.

 Now, assume $i > 0$ and that $\h_{\ib}^j(M,N)$ is Artinian and $\ia$- and $\ib$-cofinite, for $j < i$. Consider the exact sequence
$$\h_{\ib}^i (M, \Gamma_{\ib}(N)) \ \rightarrow \ \h_{\ib}^i(M,N) \ \rightarrow  \ \h_{\ib}^i \left( M, N / \Gamma_{\ib}(N) \right).$$
By Lemma \ref{LEMMA}, $\Gamma_{\ib} (N)$ is Artinian and $\ia$- and $\ib$-cofinite. Since $M$ is finitely generated, Proposition \ref{Lema3N} and Lemma \ref{Rem2.6N} imply that $\h_{\ib}^i(M, \Gamma_{\ib}(N))$ is Artinian and $\ia$- and $\ib$-cofinite. Therefore $\h_{\ib}^i(M,N)$ is Artinian and $\ia$- and $\ib$-cofinite if and only if $\h_{\ib}^i \left( M, N / \Gamma_{\ib}(N) \right)$ is Artinian and $\ia$- and $\ib$-cofinite. Hence we may assume that  $\Gamma_{\ib}(N) = 0$ and so, by Lemma \ref{Obs2N}, there is an element $r \in \ib$ which is $N$-regular.

The exact sequence
$$0 \ \rightarrow \ N \ \stackrel{\cdot r}{\rightarrow} \ N \ \rightarrow \ N / rN \ \rightarrow \ 0$$
induces the following exact sequence
$$\h_{\ib}^{i-1}\left( M, N / rN \right) \ \stackrel{}{\longrightarrow} \ \h_{\ib}^i (M,N) \ \stackrel{\cdot r}{\longrightarrow} \ \h_{\ib}^i(M,N).$$
Since $N / rN$ is $\ia$-weakly finite, by induction $\h_{\ib}^{i-1} \left( M, N / rN \right)$ is Artinian and $\ia$- and $\ib$-cofinite. Thus, by Lemma \ref{OBS}, $(0 :_{\h_{\ib}^i(M,N)} r)$ is Artinian and $\ia$- and $\ib$-cofinite. Therefore, by Proposition \ref{Obs3N}, $\h_{\ib}^i(M,N)$ is Artinian and $\ia$- and $\ib$-cofinite.
\end{proof}

\begin{cor}
	Let $(R, \m)$ be a local ring. Let $M$ be a finitely generated $R$-module and $N$ be a weakly finite $R$-module over $M$. Then $\h_{\m}^i (M, N)$ is Artinian for all $i \in \N_0$. 
\end{cor}

\begin{cor} [{\cite[Theorem 2.1]{B}}]
	Let $(R, \m)$ be a local ring and let $N$ be a weakly finite $R$-module. Then $\h_{\m}^i (N)$ is Artinian for all $i \in \N_0$. 
\end{cor}

\begin{cor} \label{Cor2.12}
Suppose that $\ia \subseteq \ib$ and $\dim_R R/\ib = 0$. Let $M$ be a finitely generated $R$-module and $N$ be an $\ia$-weakly finite $R$-module over $M$.  Then $\Ext_R^j(R/\ia,\h_{\mathfrak{b}}^i(M,N))$ and $\Ext_R^j(R/\ib,\h_{\mathfrak{b}}^i(M,N))$ have finite length for all $i$ and $j$ in $\N_0$.
\end{cor}
\begin{proof}
By Theorem \ref{Teo}, $\h_{\ib}^i(M,N)$ is Artinian and $\ia$- and $\ib$-cofinite. So, $\Ext_R^j(R/\ia,\h_{\mathfrak{b}}^i(M,N))$ and $\Ext_R^j(R/\ib,\h_{\mathfrak{b}}^i(M,N))$ are finitely generated and, since $R/\ia$ and $R/\ib$ are finitely generated $\Ext_R^j(R/\ia,\h_{\mathfrak{b}}^i(M,N))$ and $\Ext_R^j(R/\ib,\h_{\mathfrak{b}}^i(M,N))$ are Artinian by Proposition \ref{Lema3N}. Therefore, $\Ext_R^j(R/\ia,\h_{\mathfrak{b}}^i(M,N))$ and $\Ext_R^j(R/\ib,\h_{\mathfrak{b}}^i(M,N))$ have finite length.
\end{proof}

As a consequence of this corollary, we have the following result:

\begin{cor}
Let $(R, \m)$ be a local ring. Suppose $\ia \subseteq \ib$ such that $\dim_R R/\ib = 0$. Let $M$ be a finitely generated $R$-module and $N$ be a $\ia$-weakly finite $R$-module over $M$. Then $\Att_R(\Ext_R^j(R/\ia,\h_{\mathfrak{b}}^i(M,N))) = \{ \m \} = \Att_R(\Ext_R^j(R/\ib,\h_{\mathfrak{b}}^i(M,N)))$, for all $i$ and $j$ in $\N_0$.
\end{cor}
\begin{proof}
We will show  that $\Att_R(\Ext_R^j(R/\ia,\h_{\mathfrak{b}}^i(M,N))) = \{ \m \}$; the other one is analogous. By Corollary \ref{Cor2.12}, $\Ext_R^j(R/\ia,\h_{\mathfrak{b}}^i(M,N))$ has finite length, so $\Supp_R(\Ext_R^j(R/\ia,\h_{\mathfrak{b}}^i(M,N))) = \{\m\}$.

Now, by \cite[Proposition 2.9 (2) and (3)]{ooish}, $\Att_R(\Ext_R^j(R/\ia,\h_{\mathfrak{b}}^i(M,N))) \subseteq \{ \m \}$. Therefore $\Att_R(\Ext_R^j(R/\ia,\h_{\mathfrak{b}}^i(M,N))) = \{ \m \}$, since $\Att_R(\Ext_R^j(R/\ia,\h_{\mathfrak{b}}^i(M,N))) \neq \emptyset$.
\end{proof}

\subsection*{The local case}
From now on, in this section, we will assume that our ring $R$  is local with maximal ideal $\mathfrak{m}$.

\begin{thm}\label{topartcof}
Let $M$ be a finitely generated $R$-module such that $\pdim M = d < \infty$ and let $N$ be a $\ib$-weakly finite $R$-module over $M$ such that $\dim_R N = n < \infty$. Suppose $\ia \subseteq \ib$. Then $\h_{\ia}^{d+n}(M, N)$ is Artinian and $\ib$-cofinite.

In particular, if $\ia = \ib$, $\h_{\ia}^{d+n}(M, N)$ is Artinian and $\ia$-cofinite.
\end{thm}
\begin{proof}
	If $\ia = \m$ just use Theorem \ref{Teo}.
	
	Now assume $\ia \neq \m$ and choose $x \in \m \setminus \ia$. By \cite[Lemma 3.1]{divaani}, there is an exact sequence
$$\h_{\mathfrak{\ia}+xR}^{d+n} (M,N)\longrightarrow \h_{\mathfrak{\ia}}^{d+n}(M,N)\longrightarrow \h_{\mathfrak{\ia} R_x}^{d+n}(M_x,N_x),$$

\noindent where $N_x$ is the localization of $N$ at $\{x^i \mid i \geq 0\}$.  Note that  $\dim_{R_x} (N_x ) < n$ (since $\m_x = 0$), so $\h_{\mathfrak{\ia} R_x}^{d+n}(M_x, N_x)=0$. Thus there is an epimorphism $\h_{\mathfrak{\ia}+xR}^{d+n} (M, N)\longrightarrow \h_{\mathfrak{\ia}}^{d+n}(M,N)$. Now assuming ${\m} = {\ia}+(x_1,\dots ,x_r)$ and repeating this argument, we get the surjection 
$$\h_{\mathfrak{\m}}^{d+n} (M,N) \twoheadrightarrow \h_{\mathfrak{\ia}}^{d+n}(M,N).$$
Therefore $\h_{\ia}^{d+n}(M, N)$ is Artinian and $\ib$-cofinite, since $\h_{\m}^{d+n}(M,N)$ is Artinian and $\ib$-cofinite from Theorem \ref{Teo}.
\end{proof}

\begin{cor} \label{CorNew}
Let $N$ be a $\ib$-weakly finite $R$-module such that $\dim_R N = n < \infty$. Suppose $\ia \subseteq \ib$. Then $\h_{\ia}^{n}(N)$ is Artinian and $\ib$-cofinite.

In particular, if $\ia = \ib$, $\h_{\ia}^{n}(N)$ is Artinian and $\ia$-cofinite.
\end{cor}


The next result give us an important isomorphism using a $\ia$-weakly finite module. By Corollary \ref{CorNew}, the proof  follows analogously to \cite[Proposition 2.2]{yanchu}.


\begin{prop} \label{TheoNew}
Let $M$ and $N$ be $R$-modules such that $M$ is finitely generated and $N$ is $\ia$-weakly finite over $M$. Assume $\pdim M = d < \infty$ and $\dim_R N = n < \infty$. Also, assume $n \in \N$ and there exists an $\ia$ -filter regular sequence $x_1, \ldots, x_n$ on $N$. Then

\textbf{1.} $\h_{\ia}^{d+n}(M,N) \cong \Ext_{R}^{d}(M, \h_{\ia}^{n}(N))$.

\textbf{2.} $\Att_R(\h_{\ia}^{d+n}(M,N)) \subseteq \Att_R(\h_{\ia}^{n}(N))$.
\end{prop}

In \cite[Theorem B]{DiYa}, M. T. Dibaei and S. Yassemi  have shown that, if $\dim_R N = \dim_R R = n$, then $\h^n_{\ia}(N)$ has a secondary representation and
\begin{eqnarray} \label{inclusao}
\Att_R (\h_{\ia}^{n}(N)) \subseteq \{ \p \in \ass_R (N) \mid \cdim (\ia, R/\p) = n \}.
\end{eqnarray}

\begin{prop}
Let $M$ and $N$ be $R$-modules such that $M$ is finitely generated and $N$ is weakly finite. Assume $\pdim M = d < \infty$, $\dim_R R = \dim_R N = n$ and $0< n < \infty$. Also, assume $n \in \N$ and there exists an $\ia$-filter regular sequence $x_1, \ldots, x_n$ on $N$. If $\h^{d+n}_{\m}(M, N)\neq 0$, then it is not finitely generated.
\end{prop}
\begin{proof}
As $\h^{d+n}_{\m}(M, N)\neq 0$, then $\Att_R(\h^{d+n}_{\m}(M, N)) \neq \emptyset$. By item 2 of Proposition \ref{TheoNew} and (\ref{inclusao}), 
$$\Att_R(\h_{\m}^{d+n}(M,N)) \subseteq \Att_R(\h_{\m}^{n}(N)) \subseteq \{ \p \in \ass_R (N) \mid \dim_R R/\p = n \}.$$
Since $n > 0$, we have $\Att_R(\h_{\m}^{d+n}(M,N)) \nsubseteq \{ \m \}$. Since $\h_{\m}^{d+n}(M,N)$ is Artinian, it follows that $\h_{\m}^{d+n}(M,N)$ is not finitely generated by \cite[Corollary 7.2.12]{grot}.
\end{proof}

%


\section{Attached primes of the top generalized local cohomology modules}

In this section,  we assume that $(R,\m)$ is a commutative Noetherian local ring with maximal ideal $\mathfrak{m}$. The main purpose of this section is to prove \cite[Theorem B]{DiYa}, stated above, for generalized local cohomology modules. In order to do this, we give some preliminaries result. 

\begin{lem} \label{Lemma1}
	Let $M$ be finitely generated $R$-module and $N = \displaystyle \varinjlim_j N_j $, where $\{ N_j \mid  j \in J \}$ is a family of finitely generated $R$-modules. Then 
	$$\h_{\ia}^{i}(M,N) = \displaystyle \varinjlim_j \h_{\ia}^i(M, N_j), \ \mbox{ for all } i \in \N_0.$$
\end{lem}
\begin{proof}
	Since $M$ is finitely generated, $M/\ia^n M$ is finitely generated. Then, by \cite[Corollary of Theorem 1]{brow}, it is possible to commute direct limit and the functor $\Ext_R^i(M/\ia^n M, -)$. Therefore  
	\begin{eqnarray*}
	\h_{\ia}^i (M,N) &=& \displaystyle \varinjlim_n \Ext_R^i(M/\ia^n M, N) \\
	                &=& \displaystyle \varinjlim_n \Ext_R^i(M/\ia^n M, \displaystyle \varinjlim_j N_j) \\
	                &=& \displaystyle \varinjlim_j (\displaystyle \varinjlim_n \Ext_R^i(M/\ia^n M, N_j)) \\
	                &=& \displaystyle \varinjlim_j \h_{\ia}^i(M,N_j).
	\end{eqnarray*}
\end{proof}

\begin{lem} [{\cite[Lemma 1.2]{rush}}] \label{Lemma1.2}
If $M$ is an $R$-module with $\Att_R (M) = \{ \p \}$ where $\p$ is a minimal prime of $R$, then $M$ is secondary. 
\end{lem} 

\begin{prop} \label{Lemma2}
Let $M$ and $N$ be $R$-modules such that $M$ is finitely generated, $\pdim M = d < \infty$ and $\dim_R N = n < \infty$. Suppose $\Ass_R (L)$ is a finite set and $\Att_R (\h_{\ia}^{d+n}(M,L)) \subseteq \Ass_R (L)$ for every submodule $L$ of $N$ (in particular for $L=N$). Also, suppose $\Att_R (\h_{\ia}^{d+n}(M,N)) \neq \emptyset$. Then $\h_{\ia}^{d+n}(M,L)$ has a secondary representation, for all $R$-submodule $L \subseteq N$.
\end{prop}
\begin{proof}
	Let $\Ass_R (N) = \{ \p_1, \ldots, \p_l \}$. We proceed by induction on $l$.
    
    If $l=1$, then $\emptyset \neq \Att_R (\h_{\ia}^{d+n}(M,N)) \subseteq \{\p_1\}$. Therefore $\h_{\ia}^{d+n}(M,N)$ is a $\p_1$-secondary $R$-module, by Lemma \ref{Lemma1.2}.
    
    Now assume $l>1$. By \cite[Proposition 2 - p.263]{Bourb}, for each $i$, $1 \leq i \leq l$, there exist $L_i \subseteq N$ submodule such that $\Ass_R (L_i) = \{\p_i\}$ and $\Ass_R (N/L_i) = \Ass_R (N) \setminus \{\p_i\}$. By hypothesis, 
    $$\Att_R (\h_{\ia}^{d+n}(M,L_i)) \subseteq \Ass_R (L_i) = \{\p_i\} \quad \mbox{\ and } \quad \Att_R (\h_{\ia}^{d+n}(M,N/L_i)) \subseteq \Ass_R (N) \setminus \{\p_i\}.$$
Thus $\Att_R (\h_{\ia}^{d+n}(M,L_i)$ is $\p_i$-secondary or zero and $\Att_R (\h_{\ia}^{d+n}(M,N/L_i))$ has secondary representation by induction.

By the exact sequence
$$\h_{\ia}^{d+n}(M,L_i) \stackrel{\varphi_i}{\longrightarrow} \h_{\ia}^{d+n}(M,N) \longrightarrow \h_{\ia}^{d+n}(M,N/L_i) \longrightarrow 0,$$
we obtain that $\varphi_i(\h_{\ia}^{d+n}(M,L_i))$ is $\p_i$-secondary or zero. If $\varphi_i(\h_{\ia}^{d+n}(M,L_i)) = 0$ for some $i$, then $\h_{\ia}^{d+n}(M,N) \cong \h_{\ia}^{d+n}(M,N/L_i)$ which has secondary representation. So we may assume $\varphi_i(\h_{\ia}^{d+n}(M,L_i)) \neq 0$ for all $i = 1, \ldots, l$. Hence 
\begin{eqnarray*}
\Att_R (\h_{\ia}^{d+n}(M,N))/\displaystyle \sum_{i=1}^{l} \varphi_i(\Att_R (\h_{\ia}^{d+n}(M,L_i)))  &\subseteq& \bigcap_{i=1}^{l} \Att_R (\h_{\ia}^{d+n}(M,N)/\varphi_i(\h_{\ia}^{d+n}(M,L_i))) \\
&=& \bigcap_{i=1}^{l} \Att_R(\h_{\ia}^{d+n}(M,L_i)) = \emptyset.
\end{eqnarray*}
Therefore 
$$\h_{\ia}^{d+n}(M,N) = \sum_{i=1}^{l} \varphi_i (\h_{\ia}^{d+n}(M,L_i))$$
which is secondary representation for $\h_{\ia}^{d+n}(M,N)$.
\end{proof}

The next theorem is the main result of this section.

\begin{thm} 
	Let $M$ be a finitely generated $R$-module such that $\pdim M = d < \infty$. Let  $\dim_R N = \dim_R R = n$. Then $\h_{\ia}^{d+n}(M,N)$ has secondary representation and
	$$\Att_R (\h_{\ia}^{d+n}(M,N)) \subseteq \{ \p \in \ass_R (N) \mid \cdim (\ia, M, R/\p) =d+n \}.$$
 
\end{thm}
\begin{proof}
	Firstly we will show the inclusion of the set of attached primes.
We can assume $\h_{\ia}^{d+n}(M,N) \neq 0$.
	
	Let $X = \{ \p \in \Ass_R (N) \mid \cdim (\ia, M, R/\p) = d+n \}$. By \cite[Proposition 2 - p.263]{Bourb}, there exists a submodule $L \subseteq N$ such that $\Ass_R(N/L) = X$ and $\Ass_R(L) = \Ass_R(N) \setminus X$.
	
	Consider the exact sequence
	$$\h_{\ia}^{d+n}(M,L) \longrightarrow \h_{\ia}^{d+n}(M,N) \longrightarrow \h_{\ia}^{d+n}(M,N/L) \longrightarrow \h_{\ia}^{d+n+1}(M,L).$$
	Then $\h_{\ia}^{d+n+1}(M,L) = 0$, since $n+1 > \dim_R L$. On the other hand, by Lemma \ref{Lemma1}, $\h_{\ia}^{d+n}(M,L) = \displaystyle \varinjlim_j \h_{\ia}^{d+n}(M, L_j)$, where $\{ L_j \mid j\in I \}$ is a family of finite submodules of $L$.
	
	We claim that $\h_{\ia}^{d+n}(M,L_j) = 0$, for all $j \in I$.
If $\dim_R L_j = l_j < n$, then $d+n > d+l_j$ and therefore $\h_{\ia}^{d+n}(M,L_j) = 0$.	
Now, if $\dim_R L_j = n$, then
	$$\Att_R (\h_{\ia}^{d+n}(M,L_j)) = \{ \p \in \Ass_R(L_j) \mid \cdim (\ia, M, R/\p) = d+n \}.$$
	On the other hand, as $L_j$ is a submodule of $L$ and $\Ass_R(L) = \Ass_R(N) \setminus X$, we have $\Ass_R (L_j) \cap X = \emptyset$. Then, $\h_{\ia}^{d+n}(M,L_j) = 0$ and we can conclude $\h_{\ia}^{d+n}(M,L) = 0$.
	
	Therefore, $\h_{\ia}^{d+n}(M,N) \cong \h_{\ia}^{d+n}(M,N/L)$ and hence we may assume that $L = 0$ and $\Ass_R(N) = X$.
	
	Now, we claim $\Att_R(\h_{\ia}^{d+n}(M,N)) \subseteq \Ass_R(N)$.
If $r \notin \displaystyle\bigcup_{\p \in \Ass(N)} \p$, then the short exact sequence
	$$0 \longrightarrow N \stackrel{\cdot r}{\longrightarrow} N \longrightarrow N/rN \longrightarrow 0$$
induces the exact sequence
	$$\h_{\ia}^{d+n}(M,N) \stackrel{\cdot r}{\longrightarrow} \h_{\ia}^{d+n}(M,N) \longrightarrow \h_{\ia}^{d+n}(M,N/rN).$$
	Since  $\cdim(\ia, M, N/rN) < n+d$, we have $\h_{\ia}^{d+n}(M,N/rN) = 0$. Then $r \h_{\ia}^{d+n}(M,N) = \h_{\ia}^{d+n}(M,N)$ and $r \notin \displaystyle\bigcup_{\p \in \Att_R(\h_{\ia}^{d+n}(M,N))} \p$. Therefore $\displaystyle\bigcup_{\p \in \Att_R(\h_{\ia}^{d+n}(M,N))} \p \subseteq \displaystyle\bigcup_{\p \in \Ass_R (N)} \p$.
	
	We claim that $\Ass_R(N) \subseteq \Ass_R(R)$, and therefore $|\Ass_R(N)| < \infty$. To see this, let $\p \in \Ass_R(N)$ then
	$$n+d = \cdim(\ia, M, R/\p) \leq \pdim (M) + \dim_R R/\p \leq \pdim(M) + \dim_R R = n+d,$$
	which implies $\dim_R R/\p = n$. Thus $\p$ is minimal over the ideal $0$ and $\p \in \Ass_R(R)$, which is a finite set, since $R$ is Noetherian.
	
	If now $\p \in \Att_R(\h_{\ia}^{d+n}(M,N))$, then $\p \subseteq \q$, for some $\q \in \Ass_R (N)$. On the other hand, since
	$$\cdim(\ia, M, R/\p) \leq \pdim (M) + \dim_R R/\p \leq \pdim (M) + \dim_R R = d+n$$
	and
	$$n+d = \cdim(\ia, M, R/\q) \leq \cdim(\ia, M, R/\p) \leq n+d$$
	we have $\p=\q$. Therefore, $\Att_R(\h_{\ia}^{d+n}(M,N)) \subseteq \Ass_R(N)$.
    
  The secondary representation of  $\h_{\ia}^{d+n}(M,N)$  follows immediately from Proposition \ref{Lemma2} and  the first part of this proof.  
\end{proof}


The inclusion shown in the previous theorem is strict in general (see \cite[Example 6]{DiYa}). But, if $N$ is finitely generated the equality holds \cite[Theorem 2.3]{yanchu}.


\section{Torsion and Extension Functors: Some Preparatory Results}

Let $R$ be a commutative Noetherian ring. Now, we are interested in the cofiniteness of torsion and extension functors. First we will recall some definitions that will be used in the rest of the paper.

\begin{itemize}
\item[(i)] An $R$-module $M$ is called minimax, if there is a finitely generated submodule $N$ of $M$ such that $M/N$ is Artinian \cite{Zo1}.

\item[(ii)] An $R$-module $M$ is said to be $\mathfrak{a}$-cominimax   if ${\rm Supp}(M)\subseteq V(\mathfrak{a})$ and ${\rm Ext}^j_R(R/\mathfrak{a}, M)$ is minimax for all $j$ \cite{azami}.

\item[(iii)] An $R$-module $M$ is said to be 
weakly Laskerian, if the set of associated primes of any quotient module of $M$ is finite \cite{divmafi2}.

\item[(iv)] An $R$-module $M$ is said to be $\mathfrak{a}$-weakly cofinite if ${\rm Supp}(M)\subseteq V(\mathfrak{a})$ and ${\rm Ext}^j_R(R/\mathfrak{a}, M)$ is weakly Laskerian for all $j$ \cite{divmafi2}. 

\end{itemize}

Now we will recall some facts of minimax modules.

\begin{rem}\label{rem11}The following statements hold: 
\begin{itemize}
\item[(i)] Noetherian modules are minimax, as are Artinian modules.

\item[(ii)] The set of associated primes of any minimax $R$-module is finite. 

\end{itemize}
\end{rem}

In what follows, we will use Serre subcategories of the category of $R$-modules. We denote by $\mathcal{S}$ a Serre subcategory of the category of $R$-modules. Recall that the classes of Noetherian modules, of Artinian modules, of minimax modules, of weakly Laskerian and of Matlis reflexive are examples of Serre subcategories. 

\begin{lem} \label{LemmaEQ}
	Let $M$ be an $R$-module such that $M/\ia M \in \s$. Then $M/\ia^n M \in \s$ for all $n \in \N$.
\end{lem}
\begin{proof}
	We use induction on $n$. If $n=1$, then it is true by hypothesis. Now let $n>1$ and suppose the result is true for $n-1$. Since $M/\ia^{n-1}M \in \s$, $(M/\ia^{n-1}M)^k \in \s$, for all $k \in \N_0$. There is an exact sequence
	$$(M/\ia^{n-1}M)^t \stackrel{f}{\longrightarrow} M/\ia^n M \stackrel{g}{\longrightarrow} M/\ia M \longrightarrow 0$$
	where $\ia = (x_1, \ldots, x_t)$ and
	$$f(m_1 + \ia^{n-1}M, \ldots, m_t + \ia^{n-1}M) = x_1 m_1 + \ldots + x_t m_t + \ia^n M.$$
	Therefore, $M/\ia^nM \in \s$.
\end{proof}

Let $\underline{x} = x_1, \ldots, x_t$. In the next two Theorems we will use Koszul complexes, $K_{\bullet}(\underline{x})$; Koszul homology, $\h_q(\underline{x};M) = \h_q(K_{\bullet}(\underline{x}) \otimes_R M)$; and Koszul cohomology, $\h^q(\underline{x};M) = \h^q(\Hom_R(K_{\bullet}(\underline{x}), M))$. The reader can see \cite{Weibel} for more details. 

\begin{thm} \label{TheoEQ1}
	Let $M$ be an $R$-module such that $\Ext_R^i(R/\ia, M) \in \s$ for all $i \in \N_0$. Then $M/\ia^n M \in \s$ for all $n \in \N$.
\end{thm}
\begin{proof}
	In view of Lemma \ref{LemmaEQ}, it is enough to prove that $M/\ia M \in \s$.
	
	Let $\ia = (x_1, \ldots, x_t)$ and $\underline{x} = x_1, \ldots, x_t$. Then $M/\ia M \cong \h^t(\underline{x};M)$ and $\h^j(\underline{x};M) = Z^j/B^j$, where $B^j$ and $Z^j$ are the modules of coboundaries and cocycles of the complex $\Hom_R(K_{\bullet}(\underline{x}), M)$, respectively, and where $K_{\bullet}(\underline{x})$ is the Koszul complex on $\underline{x}$.
	
	Put
	$$\mathcal{C} = \{ N \mid \Ext_R^i(R/\ia,N) \in \s \mbox{ for all } i \in \N_0 \}.$$
	
	Our claim is that $B^j \in \mathcal{C}$ for all $j= 0, 1, \ldots, t$. We will prove this by induction on $j$. If $j=0$, $B^0 = 0 \in \mathcal{C}$. 
	
	Now, assume that $B^l \in \mathcal{C}$. Put $C^j = \Hom_R(K_{j}(\underline{x}), M) / B^j$. Since $K_l(\underline{x})$ is a finite free $R$-module, it follows that $\Hom_R(K_l(\underline{x}),M) \in \mathcal{C}$. Now, since $B^l \in \mathcal{C}$, we have that $C^l \in \mathcal{C}$. Hence $(0 :_{C^l} \ia) \cong \Hom_R(R/\ia, C^l) \in \s$.
	
	Because of $\ia \h^l(\underline{x};M) = 0$, it follows that $\h^l(\underline{x};M) \subseteq (0 :_{C^l} \ia)$, and so $\h^l(\underline{x};M) \in \s$. Consequently. from the short exact sequence
	$$0 \longrightarrow \h^l(\underline{x};M) \longrightarrow C^l \longrightarrow B^{l+1} \longrightarrow 0$$
	we deduce that $B^{l+1} \in \mathcal{C}$.
	
	Hence by induction we have proved that $B^j \in \mathcal{C}$ for all $j \in \N_0$.
	
	Now, since $B^t \in \mathcal{C}$ and $\Hom_R(K_t(\underline{x}),M) \in \mathcal{C}$, we obtain $C^t \in \mathcal{C}$. Hence $(0 :_{C^t} \ia) \cong \Hom_R(R/\ia, C^t) \in \s$. Thus $\h^t(\underline{x};M) \subseteq (0 :_{C^t} \ia)$ is in $\s$ too.
	
	Therefore $M/\ia M \in S$.	
\end{proof}

Next Theorem is a generalization of \cite[Theorem 2.1]{M1}.

\begin{thm} \label{TheoEQ2}
	Let $M$ be an $R$-module and $\ia = (x_1, \ldots, x_t)$ be an ideal of $R$. Then the following conditions are equivalent:
	\begin{itemize}
		\item[\textbf{(i)}] $\Ext_R^i(R/\ia, M) \in \s$, for all $i\in\N_0$;
		\item[\textbf{(ii)}] $\Tor^R_i(R/\ia, M) \in \s$, for all $i\in\N_0$;
		\item[\textbf{(iii)}] The Koszul cohomology modules $H^i(x_1, \ldots, x_t;M) \in \s$, for all $i \in \N_0$.
	\end{itemize}
\end{thm}
\begin{proof}
(i) $\Rightarrow$ (ii) 
Let 
$$\mathbb{F}_{\bullet}: \cdots \rightarrow F_2 \rightarrow F_1 \rightarrow F_0 \rightarrow 0$$ be a free resolution of finitely generated $R$-modules for $R/\ia$. Consider the complex $\mathbb{F}_{\bullet} \otimes_R M$; it follows that $\Tor_i^R(R/\ia, M) = Z_i/B_i$, where $B_i$ and $Z_i$ are the modules of boundaries and cycles of this new complex, respectively.

Put
$$\mathcal{C} = \{ N \mid \Ext_R^j(R/\ia,N) \in \s \mbox{ for all } j \in \N_0 \}.$$
Our claim is that $Z_i \in \mathcal{C}$ for all $i \in \N_0$. We will prove this by induction on $i$. 

If $i=0$, $Z_0 = F_0 \otimes_R M \in \mathcal{C}$, since $F_0$ is a finitely generated free $R$-module.

Now, assume that $Z_l \in \mathcal{C}$. 

Consider the short exact sequence
\begin{eqnarray}\label{(**)}
	0 \longrightarrow B_i \longrightarrow Z_i \longrightarrow \Tor_i^R(R/\ia, M) \longrightarrow 0,
\end{eqnarray}
where we can see $B_i \cong (F_{i+1}\otimes_R M) / Z_{i+1}$. Hence we obtain the exact sequence
$$Z_i/\ia Z_i \longrightarrow \Tor_i^R(R/\ia, M) \longrightarrow 0.$$
Therefore $\Tor_i^R(R/\ia, M)$ is homomorphic image of $Z_i/ \ia Z_i$, for all $i \in \N_0$.

Now, since $Z_l \in \mathcal{C}$, $\Ext_R^j(R/\ia, Z_l) \in \s$, for all $j \in \N_0$; then $Z_l/\ia Z_l \in \s$, by Theorem \ref{TheoEQ1}. Thus $\Tor_l^R(R/\ia,M) \in \s$. Therefore we deduce from (\ref{(**)}) that $B_l \in \mathcal{C}$ and so $Z_{l+1} \in \mathcal{C}$.

Hence by induction we have proved that $Z_i \in \mathcal{C}$, for all $i \in \N_0$. It follows by Theorem \ref{TheoEQ1} that $Z_i / \ia Z_i \in \s$, for all $i \in \N_0$, and therefore $\Tor_i^R(R/\ia, M) \in \s$, for all $i \in \N_0$.

(ii) $\Rightarrow$ (iii)
Let $\underline{x} = x_1, \ldots, x_t$. As $\h^i(\underline{x};M) \cong H_{t-i}(\underline{x}; M)$, it is sufficient to show that $H_i(\underline{x}; M) \in \s$, for all $i \in \N_0$.

Consider the Koszul complex
$$K_{\bullet}(\underline{x}): 0 \rightarrow K_t(\underline{x}) \rightarrow K_{t-1}(\underline{x}) \rightarrow \cdots \rightarrow K_1(\underline{x}) \rightarrow K_0(\underline{x}) \rightarrow 0.$$ 
Then $H_i(\underline{x};M) = Z_i/B_i$, where $B_i$ and $Z_i$ are the modules of boundaries and cycles of the complex $K_{\bullet}(\underline{x}) \otimes_R M$, respectively.

Put
$$\mathcal{C} = \{ N \mid \Tor_j^R(R/\ia, N) \in \s, \mbox{ for all } j \in \N_0 \}.$$

Consider the short exact sequence 
$$0 \longrightarrow B_i \longrightarrow Z_i \longrightarrow H_i(\underline{x};M) \longrightarrow 0.$$
Hence we obtain the exact sequence
$$Z_i/\ia Z_i \longrightarrow H_i(\underline{x};M) \longrightarrow 0,$$
thus $H_i(\underline{x};M)$ is homomorphic image of $Z_i/ \ia Z_i$, for all $i \in \N_0$.

Now, analogous to the proof of the implication (i) $\Rightarrow$ (ii), we can show that $Z_i \in \mathcal{C}$, for all $i \in \N_0$. Since $Z_i/\ia Z_i = \Tor_0^R(R/\ia, Z_i) \in \s$, for all $i \in \N_0$, we have $H_i(\underline{x};M) \in \s$, for all $i \in \N_0$.

(iii) $\Rightarrow$ (i)
Let 
$$\mathbb{F}_{\bullet}: \cdots \rightarrow F_2 \rightarrow F_1 \rightarrow F_0 \rightarrow 0$$ 
be a free resolution of finitely generated $R$-modules for $R/\ia$. Consider the complex $\Hom_R(\mathbb{F}_{\bullet}, M)$; it follows that $\Ext_R^i(R/\ia, M) = Z^i/B^i$, where $B^i$ and $Z^i$ are the modules of coboundaries and cocycles of this new complex, respectively.

Let $\underline{x} = x_1, \ldots, x_t$. Put
$$\mathcal{C} = \{ N \mid H^j(\underline{x};N) \in \s, \mbox{ for all } j \in \N_0 \}.$$

Consider the short exact sequence
$$0 \longrightarrow \Ext_R^i(R/\ia,M) \longrightarrow C^i \longrightarrow B^{i+1} \longrightarrow 0,$$
where $C^i = \Hom_R(F_i,M)/B^i$. Then $B^i \in \mathcal{C}$ (as in the proof of Theorem \ref{TheoEQ1}), for all $i \in \N_0$. Thus $C_i \in \mathcal{C}$, for all $i \in \N_0$.

Therefore, since
$$\Ext_R^i(R/\ia,M) \subseteq (0:_{C^i} \ia) \cong \Hom_R(R/\ia, C^i) \cong H^0(\underline{x};C^i)$$
and $H^0(\underline{x}; C^i) \in \s$, we can see that $\Ext_R^i(R/\ia, M) \in \s$, for all $i \in \N_0$.
\end{proof}

\begin{lem}\label{lema2} 
	Let $M$ be a finitely generated $R$-module and $N$ be an arbitrary module. Let $t$ be a non-negative integer such that ${\rm Tor}_i^R(M,N) \in \mathcal{S}$ for all $i\leq t$. Then ${\rm Tor}_i^R(L,N) \in \mathcal{S}$ for all $i\leq t$, whenever $L$ is a finitely generated $R$-module such that ${\rm Supp}_R (L)\subseteq {\rm Supp}_R(M)$.
\end{lem}
\begin{proof}
Since $\Supp_R(L) \subseteq \Supp_R(M)$, there exists a chain of $R$-modules
$$0 = L_0 \subset L_1 \subset \cdots \subset L_k = L$$
such that the factors $L_j/L_{j-1}$ are homomorphic images of a direct sums of finitely many copies of $M$ (by Gruson’s Theorem, \cite[Theorem 4.1]{WVasc}).

Consider the exact sequences
$$ 0 \rightarrow K \rightarrow M^n \rightarrow L_1 \rightarrow 0$$
$$ 0 \rightarrow L_1 \rightarrow L_2 \rightarrow L_2/L_1 \rightarrow 0$$
$$\vdots$$
$$ 0 \rightarrow L_{k-1} \rightarrow L_k \rightarrow L_k/L_{k-1} \rightarrow 0$$
for some positive integer $n$ and some finitely generated $R$-module $K$.

Let $i \leq t$ and $1 \leq j \leq k$. From the long exact sequence
$$\cdots \rightarrow \Tor^R_{i+1}(L_{j}/L_{j-1}, N) \rightarrow \Tor^R_{i}(L_{j-1}, N) \rightarrow \Tor^R_{i}(L_{j}, N) \rightarrow \Tor^R_{i}(L_{j}/L_{j-1}, N) \rightarrow \cdots$$
and the properties of Serre subcategories, $\Tor^R_{i}(L_{j}, N) \in \mathcal{S}$ if and only if $\Tor^R_{i}(L_{j-1}, N) \in \mathcal{S}$. Using an easy induction on $k$, it suffices to prove the case when $k = 1$.

So, consider the exact sequence mentioned above
\begin{eqnarray} \label{seq1}
0 \rightarrow K \rightarrow M^n \rightarrow L \rightarrow 0.
\end{eqnarray}
We now will use induction on $t$.

If $t=0$, we have that $L \otimes_R N$ is a homomorphic image of  $M^n \otimes_R N$ which belongs to $ \mathcal{S}$. Then $L \otimes_R N \in  \mathcal{S}$.

Now, lets assume $t>0$ and $\Tor^R_{i}(L',N) \in  \mathcal{S}$ for every finitely generated $R$-module $L'$ with $\Supp_R(L') \subseteq \Supp_R(M)$ and all $i < t$. The exact sequence (\ref{seq1}) induces the long exact sequence
$$\cdots \rightarrow \Tor^R_{i}(M^n, N) \rightarrow \Tor^R_{i}(L, N) \rightarrow \Tor^R_{i-1}(K, N) \rightarrow \cdots$$
so that, by the inductive hypothesis, $\Tor^R_{i-1}(K, N) \in \mathcal{S}$ for all $i \leq t$. On the other hand, $\Tor^R_{i}(M^n, N) \cong \displaystyle \bigoplus^n \Tor^R_{i}(M, N) \in \mathcal{S}$.

Therefore, $\Tor^R_{i}(L, N)$ for all $i \leq t$, and the proof is complete.
\end{proof}

\begin{lem}\label{lema3} 
	Let $M$ be a finitely generated $R$-module and $N$ be an arbitrary module. Let $t$ be a non-negative integer such that ${\rm Ext}^i_R(M,N) \in \mathcal{S}$ for all $i\leq t$. Then ${\rm Ext}^i_R(L,N) \in \mathcal{S}$ for all $i\leq t$, whenever $L$ is a finitely generated $R$-module such that ${\rm Supp}_R (L)\subseteq {\rm Supp}_R(M)$.
\end{lem}
\begin{proof}
The proof follows by a similar way to what was done in Lemma \ref{lema2}.
\end{proof}

The next Theorem is a key ingredient for the rest of the paper.

\begin{thm}\label{teo1} 
	Let $t$ be a non-negative integer. Then, for an  arbitrary $R$-module $N$, the following conditions are equivalent:
\begin{itemize}
\item[(i)] ${\rm Tor}^R_i(R/\ia,N)\in \mathcal{S}$ for all $i\leq t$.
\item[(ii)] For any finitely generated $R$-module $M$ with ${\rm Supp}_R (M) \subseteq V(\ia)$,  ${\rm Tor}^R_i(M,N)\in \mathcal{S}$ for all $i\leq t$.
\item[(iii)] For any $R$-ideal $\ib$ with $\ia\subseteq \ib$, ${\rm Tor}^R_i(R/\ib,N)\in \mathcal{S}$ for all $i\leq t$.
\item[(iv)]For any minimal prime $\mathfrak{p}$ over $\ia$, ${\rm Tor}^R_i(R/\mathfrak{p},N)\in \mathcal{S}$ for all $i\leq t$.
\item[(v)] ${\rm Ext}_R^i(R/\ia,N)\in \mathcal{S}$ for all $i\leq t$.
\item[(vi)] For any finitely generated $R$-module $M$ with ${\rm Supp}_R (M) \subseteq V(\ia)$,  ${\rm Ext}_R^i(M,N)\in \mathcal{S}$ for all $i\leq t$.
\item[(vii)]  For any $R$-ideal $\ib$ with $\ia\subseteq \ib$, ${\rm Ext}_R^i(R/\ib,N)\in \mathcal{S}$ for all $i\leq t$.
\item[(viii)] For any minimal prime $\mathfrak{p}$ over $\ia$, ${\rm Ext}_R^i(R/\mathfrak{p},N)\in \mathcal{S}$ for all $i\leq t$.
\end{itemize}
\end{thm}
\begin{proof}
(i)$\Rightarrow$(ii) It follows from Lemma \ref{lema2}, since $\Supp_R(R/\ia) = V(\ia)$. 

(ii)$\Rightarrow$(iii) Take $M = R/\mathfrak{b}$ and observe that $\Supp_R(R/\ib) = V(\ib) \subseteq V(\ia)$.

(iii)$\Rightarrow$(iv) Immediate.

(iv)$\Rightarrow$(i) Let $\p_1, \ldots, p_n$ be the minimal primes of $\ia$. Then, by assumption, $\Tor^R_{i}(R/\p_j, N) \in \mathcal{S}$ for all $j = 1, \ldots, n$. Hence $\bigoplus_{j=1}^{n}\Tor^R_{i}(R/\p_j, N) \cong \Tor^R_{i}(\bigoplus_{j=1}^{n} R/\p_j, N) \in \mathcal{S}$. Since $\Supp_R(R/\ia) = \Supp_R(\bigoplus_{j=1}^{n} R/\p_j)$, it follows by Lemma \ref{lema2} that ${\rm Tor}^R_i(R/\ia,N)\in \mathcal{S}$ for all $i\leq t$, as required.

(v)$\Rightarrow$(vi)$\Rightarrow$(vii)$\Rightarrow$(viii)$\Rightarrow$(v) It follows in a similar way to what was done previously, using Lemma \ref{lema3}.

(i)$\Leftrightarrow$(v) Follows by Theorem \ref{TheoEQ2}.
\end{proof}

In particular, for the class of minimax modules we obtain the following result.

\begin{cor}
Let $\mathfrak{a}$ be an ideal of $R$ such that $\dim_R R/\mathfrak{a}=1$, and let $t$ be a non-negative integer. Then, for an  arbitrary $R$-module $N$, the following conditions are equivalent:
\begin{itemize}
\item[(i)] ${\rm Ext}_R^i(R/\mathfrak{a},N)$ is minimax for all $i\leq t$.

\item[(ii)] ${\rm Tor}^R_i(R/\mathfrak{a},N)$ is minimax for all $i\leq t$.

\item[(iii)] The Bass number $\mu^i(\mathfrak{p},N)$ is finite for all $\p \in V(\mathfrak{a})$ for all $i\leq t$.

\item[(iv)] The Betti number $\beta^i(\mathfrak{p},N)$ is finite for all $\p \in V(\mathfrak{a})$ for all $i\leq t$. 

\item[(v)] $H^i_\mathfrak{a}(N)$ is $\mathfrak{a}$-cominimax, for all integer $i$.
\end{itemize}
\end{cor}
\begin{proof}
The proof follows by Theorem \ref{teo1}, \cite[Theorem 3.3]{irani} and \cite[Corollary 2.5]{abasi1}.
\end{proof}

\begin{prop}\label{lemma2.5} Let $M$ be a minimax $R$-module. If there is $n\in \mathbb{N}$ and  $\m_1,\ldots,\m_s$ maximal ideals of $R$ such that $\m_1\m_2\cdots\m_s M=0$, then $M$ has finite length.
\end{prop}
\begin{proof}
Consider the short exact sequence
$0\rightarrow L\rightarrow M \rightarrow K\rightarrow 0$, where $L$ is finitely generated and $K$ is an Artinian $R$-module. Note that $\m_1 \m_2 \cdots \m_s L = 0$, then $L$ is Artinian, since we know that $L$ is finitely generated. Thus $M$ is an Artinian module and $\m_1\m_2\cdots\m_s M=0$. Therefore, $M$ has finite length.
\end{proof}

\section{Cofiniteness of torsion and Extension functors} 

Let $R$ be a commutative Noetherian local ring. Our purpose in this section is to give some answers about Question 3 in the introduction. 

The reader can compare the  next result with \cite[Theorem 2.3]{kubik}.

\begin{thm}\label{kubik}
	Let $(R, \m)$ be a commutative Noetherian local ring. Let $H$ and $N$ be $R$-modules such that $H$ is Artinian and $\mathfrak{a}$-cofinite and $N$ is minimax. Then, for each $i \geq 0$, the module $\Ext_R^i(N,H)$ is minimax and $\widehat{\mathfrak{a}}$-cofinite over $\widehat{R}$. 
\end{thm}
\begin{proof}
Since $N$ is minimax, there exists submodule $L \subseteq N$ such that $L$ is finitely generated and $N/L$ is Artinian. We get a long exact sequence
$$\cdots  \rightarrow \Ext_R^i(N/L,H) \rightarrow \Ext_R^i(N,H) \rightarrow \Ext_R^i(L,H) \rightarrow \cdots.$$

By Proposition \ref{lema1}, since Artinian and $\mathfrak{a}$-cofinite modules form a Serre subcategory of the category of $R$-modules, we obtain that $\Ext_R^i(L,H)$ is Artinian and $\mathfrak{a}$-cofinite. Then is also Artinian and $\widehat{\mathfrak{a}}$-cofinite over $\widehat{R}$. Thus $\Ext_R^i(L,H)$ is minimax and $\widehat{\mathfrak{a}}$-cofinite over $\widehat{R}$.

On the other hand, by \cite[Corollary 2.3]{kubik}, $\Ext_R^i(N/L,H)$ is Noetherian over $\widehat{R}$. Thus $\Ext_R^i(N/L,H)$ is minimax and $\widehat{\mathfrak{a}}$-cofinite over $\widehat{R}$.

Therefore, since the class of  $\widehat{\mathfrak{a}}$-cofinite  minimax modules is a Serre subcategory \cite[Corollary 4.4]{M1}, $\Ext_R^i(N,H)$ is minimax and $\widehat{\mathfrak{a}}$-cofinite over $\widehat{R}$.
\end{proof}

The next two results partially answer  a generalization of Hartshorne's conjecture.
\begin{cor}
Let $(R,\mathfrak{m})$ be a complete local ring and $\ia$ an $R$-ideal such that $\dim_R R/\ia = 0$. Let $M$ be a finitely generated $R$-module, and let $N$ be an $\ia$-weakly finite $R$-module over $M$. 
Then, for any minimal $R$-module $L$ and for each $i\geq 0$ and $j\geq 0$,  $\Ext_R^i(L,\h_{\ia}^{j}(M, N))$ is minimax and $\mathfrak{a}$-cofinite  $R$-module.
\end{cor}
\begin{proof}
The proof follows by Theorem \ref{Teo} and Theorem \ref{kubik}.
\end{proof}

\begin{cor}
Let $(R,\mathfrak{m})$ be a complete local ring. Let $M$ be a finitely generated $R$-module such that $\pdim M = d < \infty$ and let $N$ be an $\ia$-weakly finite $R$-module over $M$ such that $\dim_R N = n < \infty$. 
Then, for each $i\geq 0$,  $\Ext_R^i(L,\h_{\ia}^{d+n}(M, N))$ is minimax and $\mathfrak{a}$-cofinite  $R$-module, for any minimax module $L$.
\end{cor}
\begin{proof}
The proof follows by Theorem \ref{topartcof} and Theorem \ref{kubik}.
\end{proof}

\begin{lem}\label{finitelength} Let $N$ be a nonzero $\mathfrak{a}$-cominimax $R$-module, where $\ia$ is an ideal of $R$. Then, for any nonzero $R$-module $M$ of finite length, the $R$-module ${\rm Tor}^R_i(M,N)$ is minimax and has finite length for all $i\geq 0$.
\end{lem}
\begin{proof} First, note that ${\rm Supp}_R(M)$ is a finite non-empty subset of the set of all maximal ideals  of $R$. Let ${\rm Supp}_R(M)=\{\m_1,\ldots, \m_r\}$ and $\mathfrak{b}= \m_1\m_2\ldots\m_r.$ As ${\rm Supp}_R(M) = V(\ib)$, by Lemma \ref{lema2}, it is sufficient to show that ${\rm Tor}^R_i(R/\mathfrak{b},N)$ has finite length for all $i\geq 0$. By the isomorphism ${\rm Tor}^R_i(R/\mathfrak{b},N)\cong \bigoplus_{j=1}^r {\rm Tor}^R_i(R/\m_j,N)$, it is enough to show that ${\rm Tor}^R_i(R/\m_j,N)$ has finite length for all $i \in \N_0$ and $j = 1, \ldots, r$. 

Fix $j$ and let $i \geq 0$ be an integer such that ${\rm Tor}^R_i(R/\m_j,N)\neq 0$. Note that $\m_j \in {\rm Supp}_R(N) \subseteq V(\ia)$ and ${\rm Ext}_R^i(R/\mathfrak{a},N)$ is minimax for all $i\geq 0$. Hence ${\rm Ext}_R^i(R/\m_j,N)$ is minimax and has finite length by Theorem \ref{teo1} and Proposition \ref{lemma2.5}. Therefore, applying Theorem \ref{teo1} again, we obtain that ${\rm Tor}^R_i(R/\m_j,N)$ is minimax and has finite length for all $i \in \N_0$ and $j = 1, \ldots, r$.
\end{proof}

\begin{cor}
Let $N$ be a nonzero minimax $R$-module. Then, for any nonzero $R$-module $M$ of finite length, the $R$-module $M \otimes_R N$ is minimax and has finite length for all $i \geq 0$.
\end{cor}
\begin{proof}
Take the ideal $\ia = 0$ in Lemma \ref{finitelength}.
\end{proof}

Now we are able to show the main result of this section.

\begin{thm}\label{cases} Let $N$ be a nonzero $\mathfrak{a}$-cominimax $R$-module and $M$ be a finitely generated $R$-module.
\begin{itemize}
\item[(i)] If $\dim_R M=1$, then the $R$-module ${\rm Tor}^R_i(M,N)$ is Artinian and $\mathfrak{a}$-cofinite for all $i\geq 0$. 

\item[(ii)] If $\dim_R M=2$, then the $R$-module ${\rm Tor}^R_i(M,N)$ is $\mathfrak{a}$-cofinite for all $i\geq 0$. 
\end{itemize}
\end{thm}
\begin{proof} $(i)$ Since $N$ is $\ia$-cominimax and ${\rm Supp}_R(\Gamma_\mathfrak{a}(M))\subseteq V(\mathfrak{a})$, we obtain that $\Tor^R_{i}(\Gamma_\mathfrak{a}(M), N)$ is minimax for all $i\geq 0$ by Theorem \ref{teo1}. Now, by the short exact sequence
$$0\rightarrow \Gamma_\mathfrak{a}(M)\rightarrow M \rightarrow M/\Gamma_\mathfrak{a}(M)\rightarrow 0,$$ we can deduce the following long exact sequence, for all $i\geq 0$
$$\cdots  \rightarrow \Tor^R_{i}(\Gamma_\mathfrak{a}(M), N) \rightarrow \Tor^R_{i}(M, N) \rightarrow \Tor^R_{i}(M/\Gamma_\mathfrak{a}(M), N)\rightarrow \Tor^R_{i-1}(\Gamma_\mathfrak{a}(M), N)\rightarrow\cdots.$$
 So it is sufficient to show that, for all $i\geq 0$,  $\Tor^R_{i}(M/\Gamma_\mathfrak{a}(M), N)$ is Artinian and $\mathfrak{a}$-cofinite. Hence, we may assume that $\Gamma_\mathfrak{a}(M)=0$ and therefore ensure the existence of \\ $x \in \mathfrak{a} \backslash \cup_{\p \in \Ass_R(M)}\p$  by Lemma \ref{Obs2N}.  
From the short exact sequence 
$$0\rightarrow M\stackrel{x}\rightarrow M \rightarrow M/xM\rightarrow 0,$$  we obtain following long exact sequence 
$$\cdots  \rightarrow \Tor^R_{i}(M/xM, N)\rightarrow \Tor^R_{i-1}(M, N) \stackrel{x}\rightarrow \Tor^R_{i-1}(M, N) \rightarrow \Tor^R_{i-1}(M/xM, N)\rightarrow\cdots.$$

By Lemma \ref{finitelength}, the $R$-module $\Tor^R_{i}(M/xM, N)$ is of finite length, for all $i\geq 0$, because $M/xM$ has finite length. So, by the long exact sequence, $(0:_{\Tor^R_{i}(M, N)}x)$ has finite length for all $i\geq 0$, and therefore $(0:_{\Tor^R_{i}(M, N)}\ia)$ has also finite length. Finally, since ${\rm Supp}_R(\Tor^R_{i}(M, N))\subseteq  \Supp_R(N)\subseteq V(\mathfrak{a})$, we can conclude that $\Tor^R_{i}(M,N)$ is $\mathfrak{a}$-torsion. Therefore for all $i\geq 0$,  $\Tor^R_{i}(M,N)$ is an Artinian $R$-module by \cite[Theorem 1.3]{MelStab}. The $\ia$-cofiniteness of $\Tor^R_{i}(M,N)$ follows by \cite[Theorem 4.3]{M1}.

$(ii)$ Proceeding similarly to the proof of item (i), we may assume that assume that $\Gamma_\mathfrak{a}(M)=0$, and so  take  $x \in \mathfrak{a} \backslash \cup_{\p \in \Ass_R(M)}\p$. The short exact sequence 
$$0\rightarrow M\stackrel{x}\rightarrow M \rightarrow M/xM\rightarrow 0,$$  induces following long exact sequence 
$$\cdots  \rightarrow \Tor^R_{i}(M/xM, N)\rightarrow \Tor^R_{i-1}(M, N) \stackrel{x}\rightarrow \Tor^R_{i-1}(M, N) \rightarrow \Tor^R_{i-1}(M/xM, N)\rightarrow\cdots.$$

Therefore $(0:_{\Tor^R_{i}(M, N)}\ia)$ and $\Tor^R_{i}(M, N)/x\Tor^R_{i}(M, N)$ are Artinian $R$-modules, by item $(i)$ and Lemma \ref{finitelength}, and $\mathfrak{a}$-cofinite, by \cite[Corollary 4.4]{M1}, for all $i\geq 0$. Therefore $\Tor^R_{i}(M, N)$ is $\mathfrak{a}$-cofinite for all $i \geq 0$, by \cite[Corollary 3.4]{M1}. 
\end{proof}

\begin{cor} Let $\ia$ be an $R$-ideal such that $\dim_R R/\ia = 0$. Let $L, M$ be  finitely generated $R$-modules and let $N$ be a $\ia$-weakly finite $R$-module over $M$. Then:  
\begin{itemize}
\item[(i)] If  $\dim_R L=1$, then ${\rm Tor}^R_i(L,H_{\ia}^{j}(M,N))$ is Artinian and $\mathfrak{a}$-cofinite for all $i, j\geq 0$.
\item[(ii)] If  $\dim_R L=2$, then ${\rm Tor}^R_i(L,H_{\ia}^{j}(M,N))$ is $\mathfrak{a}$-cofinite for all $i, j\geq 0$.
\end{itemize}
\end{cor}
\begin{proof}
The result follows by Theorem \ref{Teo} and Theorem \ref{cases}.
\end{proof}

\begin{prop}\label{dimM} Let $N$ be a nonzero $\mathfrak{a}$-cominimax $R$-module and $M$ be a finitely generated $R$-module. If $\dim_R N\leq 1$, then the $R$-module ${\rm Tor}^R_i(M,N)$ is $\mathfrak{a}$-cominimax for all $i\geq 0$. 
\end{prop} 
\begin{proof} Let $\mathbb{F}_{\bullet}$ be a resolution of $M$ consisting of finite free $R$-modules. Since ${\rm Tor}^R_i(M,N)= H_i(\mathbb{F}_{\bullet} \otimes_RN))$ is a subquotient of a finite direct sum of copies of $N$. Now the result follows by the fact that the category of $\mathfrak{a}$-cominimax modules with dimension less than or equal to $1$ is an Abelian category  \cite[Theorem 2.5]{irani}.
\end{proof}

\begin{cor}\label{cordimM}Let $N$ be a nonzero $\mathfrak{a}$-cofinite $R$-module and $M$ be a finitely generated $R$-module. If $\dim_R N\leq 1$, then the $R$-module ${\rm Tor}^R_i(M,N)$ is $\mathfrak{a}$-cofinite for all $i\geq 0$.
\end{cor}

Now we will investigate the behavior of torsion functors for larger dimensions. For this purpose, weakly Laskerian modules  are key ingredients.

\begin{rem}\label{rem1} Note that if $R$ is a Noetherian ring, $\mathfrak{a}$ is an ideal of $R$ and $M$ an $R$-module, then ${\rm Tor}^R_i(R/\mathfrak{a},M)$ is a weakly Laskerian $R$-module for all $i\geq 0$ if and only if ${\rm Ext}_R^i(R/\mathfrak{a},M)$ is weakly Laskerian $R$-module for all $i\geq 0$. 
\end{rem}

The proof of this Remark follows by Theorem \ref{teo1} and the fact that the class of weakly Laskerian modules is a Serre subcategory. 

\begin{thm}\label{casesweakly} Let $(R,\mathfrak{m})$ be a local ring and $\mathfrak{a}$ an $R$-ideal. Let $N$ be a nonzero $\mathfrak{a}$-cominimax $R$-module and $M$ be a finitely generated $R$-module. 
 Then the $R$-module ${\rm Tor}^R_i(M,N)$ is $\mathfrak{a}$-weakly cofinite for all $i\geq 0$ when one of the following cases holds:
 \begin{itemize}
 \item[(i)] $\dim_R N\leq 2$.
 
 \item[(ii)] $\dim_R M=3$.
 \end{itemize} 
\end{thm}
\begin{proof} Note that, in the of Remark \ref{rem1}, it is sufficient to show that the $R$-modules \\
$\Tor_j^R(R/\mathfrak{a},{\rm Tor}^R_i(M,N))$ are weakly Laskerian for all $i\geq 0$ and $j\geq 0$. For this purpose, consider the set $\Lambda = \{ {\rm Tor}^R_j(R/\mathfrak{a},{\rm Tor}^R_i(M,N)) \mid i\geq 0 \ \mbox{ and } \  j\geq 0 \} $. Let $K \in \Lambda$ and let $K'$ be a submodule of $K$. The proof is complete if we show that the set $\Ass_R (K/K')$ is finite.  Note that we may assume that $R$ is complete by \cite[Ex 7.7]{Mats} and \cite[Lemma 2.1]{marley1}.   
Suppose that $\Ass_R (K/K')$ is an infinite set. So, we can consider $\{\p_s \}_{s=1}^\infty$ a countably infinite subset of non-maximal elements of $\Ass_R (K/K')$. Further $\mathfrak{m} \not\subseteq \cup_{s=1}^\infty \p_s$ by \cite[Lemma 3.2]{marleyvassilev}. Define $S:= R\backslash \cup_{s=1}^\infty \p_s$. Now, we will analyze the cases (i) and (ii).

If $\dim_R N \leq 2$, then we can conclude that $S^{-1}N$ is a $S^{-1}\mathfrak{a}$-cominimax $S^{-1}R$-module of dimension at most one by \cite[Ex 7.7]{Mats} and \cite[Lemma 3.4]{HR}. Therefore $\Tor^{S^{-1}R}_i(S^{-1}M, S^{-1}N)$ is $S^{-1}\mathfrak{a}$-cominimax for all $i\geq 0$, by Proposition \ref{dimM}.

Therefore $S^{-1}K/S^{-1}K'$ is a minimax $S^{-1}R$-module and so, $\Ass_{S^{-1}R} (S^{-1}K/S^{-1}K')$ is finite by Remark \ref{rem11}. However, for each $s$, we have that $S^{-1}\p_s \in \Ass_{S^{-1}R} (S^{-1}K/S^{-1}K')$, and so we obtain  a contradiction. This completes the proof.  

In case $\dim_R M=3$, we obtain that $\Tor_{S^{-1}R}(S^{-1}M, S^{-1}N)$ is $S^{-1}\mathfrak{a}$-cominimax, by Lemma \ref{finitelength} and Theorem \ref{cases}. Now the proof  follows similarly to the one previously made.
\end{proof}

Now we give some applications of the results shown in this section.

\begin{cor} Let $N$ be a nonzero minimax $R$-module  and $M$ be a finitely generated $R$-module. 
\begin{itemize}
\item[(i)] If $\dim_R \h^i_\mathfrak{a}(N)\leq 1$ (e.g. $\dim_R N\leq 1$ or $\dim_R R/\mathfrak{a}=1$), then the $R$-module ${\rm Tor}^R_j(M,\h^i_\mathfrak{a}(N))$ is $\mathfrak{a}$-cominimax for all $i\geq 0$  and $j\geq 0$. Also, for for all $i\geq 1$ and $j\geq 0$, the $R$-module ${\rm Tor}^R_j(M,\h^i_\mathfrak{a}(N))$ is $\mathfrak{a}$-cofinite. 

\item[(ii)] If $(R,\mathfrak{m})$ is a local ring  and $\dim_R \h^i_\mathfrak{a}(M)\leq 2$ (e.g. $\dim_R R/I\leq 2$) , then the $R$-module ${\rm Tor}^R_j(M,\h^i_\mathfrak{a}(N))$ is $\mathfrak{a}$-weakly cofinite for all $i\geq 0$ and $j\geq 0$. 

\item[(iii)]  Let $L$ and $M$ be  finitely generated $R$-modules such that $\pdim M = d < \infty$ and $\dim_R L=3$, and let $N$ be an $\ia$-weakly finite $R$-module over $M$ such that $\dim_R N = n < \infty$. 
Then, for each $i\geq 0$,  $\Tor^R_i(L,\h_{\ia}^{d+n}(M, N))$ is  $\mathfrak{a}$-weakly cofinite  $R$-module.
\end{itemize}
\end{cor}
\begin{proof}
$(i)$ First note that $\h^i_\mathfrak{a}(N)$ is an $\mathfrak{a}$-cominimax $R$-module for all $i\geq 0$, by \cite[Theorem 2.2]{abasi1}. Now the first statement follows by Proposition \ref{dimM}.  The cofiniteness of ${\rm Tor}^R_j(M,\h^i_\mathfrak{a}(N))$  follows by the fact that ${\rm Ext}_R^j(M,\h^i_\mathfrak{a}(N))$ is finite for all $i\geq 1$ and $j\geq 0$  by \cite[Theorem 2.2]{abasi1} and Corollary \ref{cordimM}.

$(ii)$ The proof follows analogously to Theorem \ref{casesweakly} $(i)$, using the previous item.

$(iii)$ Apply Theorem \ref{topartcof} and Theorem \ref{casesweakly} $(ii)$.
\end{proof}

{\bf Acknowledgments}. 
The authors would like to thank Roger Wiegand for his useful comments and suggestions about this paper.

\end{document}